\documentclass[11pt]{article}

\usepackage{geometry}
\usepackage{amssymb}
\usepackage{amsmath}
\usepackage{amsthm}
\usepackage{amsopn}
\usepackage{amscd}
\usepackage{exscale}
\usepackage{graphicx}
\usepackage[all,cmtip]{xy}
\usepackage{adjustbox}

\theoremstyle{plain}
\newtheorem{thm}{Theorem}[section]
\newtheorem{prop}[thm]{Proposition}

\newtheorem{lem}[thm]{Lemma}

\theoremstyle{definition}
\newtheorem{defn}[thm]{Definition}

\theoremstyle{remark}

\newtheorem*{thma}{{\bf Theorem A}}

\newcommand{\SSS}{\mathcal{S}}
\newcommand{\C}{\mathcal{C}}
\newcommand{\QQ}{{\bar{\mathbb{Q}}}}
\newcommand{\W}{{\mathcal{W}}}
\newcommand{\CC}{{\mathbb{C}}}
\newcommand{\OO}{{\mathcal{O}}}
\newcommand{\Q}{\mathbb{Q}}

\newcommand{\R}{\mathbb{R}}
\newcommand{\Z}{\mathbb{Z}}
\newcommand{\HH}{\mathcal{H}}

\newcommand{\E}{\mathcal{E}}
\newcommand{\A}{{\mathcal{A}}}
\newcommand{\RR}{{\mathcal{R}}}

\author{Dipramit Majumdar}
\title{Endoscopic Transfer between Eigenvarieties for definite Unitary groups}
\date{}

\begin{document}
\maketitle
\noindent

\begin{abstract}
In this paper, we extend the endoscopic transfer of definite unitary group $U(n)$, which sends a pair of automorphic forms of $U(n_{1}),U(n_{2})$ to an automorphic form of $U(n_{1}+n_{2})$, to finite slope $p$-adic automorphic forms for definite unitary groups by constructing a rigid analytic map between eigenvarieties $\E_{n_{1}} \times \E_{n_{2}} \to \E_{(n_{1}+n_{2})}$, which at classical points interpolates endoscopic transfer map. 
\end{abstract}

\tableofcontents

\section{Introduction}
 Let us fix a rational prime $p$ and embeddings, 
 \begin{equation*}
 \iota_{p}: \QQ \to \QQ_{p} ,\text{ and  } \iota_{\infty}: \QQ \to \CC.
 \end{equation*}
 For any prime $\ell$, we also fix an embedding of $W_{\Q_{\ell}} \hookrightarrow W_{\Q}$, where $W_{F}$ denotes the Weil group of $F$.
Let $E/\Q$ be a quadratic imaginary field and $G/\Q$ be a unitary group in $n \geq 1$ variables attached to $E/\Q$. We assume that $G(\R)$ is the compact real unitary group (that is $G$ is definite), and that $G(\Q_{p}) \simeq GL_{n}(\Q_{p})$. We fix such a $G$ and denote it by definite unitary group $U(n)$.\\
By recent work of Clozel, Harris, Labesse, Ng{\^o} and many other mathematicians, we made significant progress in the Langlands program for the unitary group $U(n)$. In particular we know how to attach a Galois representation $\rho_{\pi}$ to an automorphic form $\pi$ of $U(n)$. Since the fundamental lemma for the unitary group is known, in principle we know in most cases the classical endoscopic transfer for $U(n)$ (for $n \leq 3$, this is due to Rogawski(\cite{MR1081540}, in general see \cite{HarrisBook1}, and volume 2). Let $n_{1}, n_{2}$ be two natural numbers and $n=n_{1}+n_{2}$, then $U(n_{1}) \times U(n_{2})$ is an endoscopic subgroup of $U(n)$. Let $\pi_{1}$ and $\pi_{2}$ be automorphic representations of $U(n_{1})$ and $U(n_{2})$ respectively. Then the endoscopic transfer of $(\pi_{1},\pi_{2})$ is an automorphic representation $\pi$ of $U(n)$. If $\rho_{\pi_{1}}, \rho_{\pi_{2}},\rho_{\pi}$ are the Galois representations associated to $\pi_{1}, \pi_{2},\pi$ respectively, then we have $\rho_{\pi} = \rho_{\pi_{1}} \oplus \rho_{\pi_{2}}$.\\
Another aspect of study of automorphic forms, initiated by Serre, is that, a prime number $p$ being chosen, the notion of classical automorphic forms can be extended to the notion of $p$-adic automorphic forms. Using ordinary $p$-adic automorphic forms, Hida constructed family of ordinary $p$-adic automorphic forms, known as Hida family, for various reductive groups. Coleman and Mazur \cite{Coleman.1998}, extended Hida's work to non-ordinary $p$-adic modular forms, and constructed tame level $1$ eigencurve. Buzzard \cite{Buzzard.2007} formalized the construction of Coleman and Mazur, and gave a machinery, called the ``Eigenvariety Machine", which gives a $p$-adic family of automorphic forms, known as the eigenvariety, as output. He used it to construct tame level $N$ eigencurve and eigenvariety for Hilbert modular forms. Using Buzzard's eigenvariety machine, Chenevier \cite{Chenevier.2004} and Emerton \cite{Emerton.2006}, seperately constructed eigenvariety for definite unitary group $U(n)$, which we denote by $\E_{n}$. The eigenvariety $\E_{n}$ is a reduced rigid analytic space, which is equidimensional of dimension $n$.\\
In this paper, we construct $p$-adic endoscopic transfer map for definite unitary group $U(n)$. Similar work was done by Chenevier \cite{MR2111512}, where he constructed the $p$-adic Jacques-Langlands correspondence for $GL_{2}$. We follow a similar approach. More precisely, let $\E_{n}$ denote the eigenvariety associated to definite unitary group $U(n)$, as constructed by Chenevier. In the eigenvariety $\E_{n}$, every automorphic form $\pi$ of $U(n)$ appears not once but (roughly) $n!$ times as a point of $\E_{n}$, each time augmented with a little combinatorial structure called a refinement $\RR$, which is an ordering of the eigenvalues of semi-simple conjugacy class associated to $\pi_{p}$. We construct a rigid analytic map 
$$f: \E_{n_{1}} \times \E_{n_{2}} \to \E_{n} $$
which at ``nice" classical points interpolate classical endoscopic transfer. Further we show that such a map is unique. 

\begin{thma}
There exists a unique map of eigenvarieties 
$$f_{(n_{1},n_{2})}: \E_{n_{1}} \times \E_{n_{2}} \to \E_{n}$$
where $n = n_{1}+ n_{2}$, such that a ``nice" classical point of $\E_{n_{1}} \times \E_{n_{2}}$, $((\pi_{1}, \RR_{1}),(\pi_{2},\RR_{2}))$ maps to a classical point of $\E_{n}$, $(\pi,\RR)$, where $\pi$ is the endoscopic transfer of $\pi_1$ and $\pi_2$ and $\RR = (\mu_{1}(Frob_{p})\RR_{1}, \mu_{2}(Frob_{p})\RR_{2})$, here $\mu_1$ and $\mu_2$ as in (\ref{mu}).
\end{thma}

\noindent We call the image of the map $f_{(n_{1},n_{2})}$ the endoscopic component of type $(n_{1},n_{2})$ in $\E_{n}$, and denote it by $\E_{(n_{1},n_{2})}$. In particular, we define $p$-adic endoscopic transfer $\pi$ for any pair of $p$-adic automorphic forms $(\pi_{1},\pi_{2})$ of the unitary groups $U(n_{1}), U(n_{2})$ as the image of $(\pi_1,\pi_2)$ under the map $f_{(n_{1},n_{2})}$.\\

\noindent Key ingredient of our proof is a comparison theorem due to Chenevier (see Theorem \ref{compthm}), which allows us to construct the map, by constructing maps between the fibers of the eigenvarieties over Zariski dense set of points in the weight space $\W$.\\

\noindent {\bf Plan for the paper:} In section 2, of the paper we briefly recall the construction of the eigenvariety $\E_{n}$, and a few of it's relevant properties. In section 3, we review some known (and expected) results about the classical endoscopic transfer of automorphic forms of definite unitary group $U(n)$. In the final section, we give the construction of the rigid analytic map between the eigenvarieties $\E_{n_{1}} \times \E_{n_{2}}$ and $\E_{n}$.  
 
 \section{ Brief review of the construction of eigenvariety for $U(n)$}
 In this section we briefly recall the construction of Eigenvariety for $U(n)$ due to  Chenevier. For details we refer readers to \cite[Chapter 7]{Bellaiche.2009} , \cite{Chenevier.2004}.

\subsection{Construction}\label{eigenconstrct}
The eigenvariety depends on the choice of Hecke algebra $H$, so first we fix a Hecke algebra for our purpose. We fix a subset $S_0$ (which have Dirichlet density 1) of primes $l$ ($l \neq p$) split in $E$, such that $G(\Q_{l}) \simeq GL_n(\Q_{l})$ and $G(\Z_{l})$ is a maximal compact subgroup. We set $\HH^{ur} = \otimes_{l \in S_{0}} \HH_{l}$, where $\HH_{l} \simeq \QQ_{p}[T_{1,l},\cdots,T_{n,l},T_{1,l}^{-1},\cdots, T_{n,l}^{-1}]^{S_n}$ is the Hecke algebra for $GL_{n}(\Q_{l})$, with respect to it's maximal compact subgroup $G(\Z_{l})$. Let $T$ be the diagonal torus of $G(\Q_{p}) \cong GL_n(\Q_{p})$, and $T^{0}= T \cap GL_n(\Z_{p})$. Let $U \simeq T/T^{0}$ be the subgroup of diagonal matrices whose entries are integral powers of $p$, $U^{-} \subset U$ be the submonoid whose elements have the form $diag(p^{a_{1}},\dots,p^{a_{n}})$, with $a_{i} \in \Z$ and $a_{i} \geq a_{i+1}$ for all $i$. Let $\A_{p}^{-}$ be the subring of Hecke-Iwahori algebra generated by $u \in U^{-}$ and $\A_{p}$ be the Atkin-Lehner algebra contained in the Hecke-Iwahori algebra of $G(\Q_{p}) \simeq GL_n(\Q_{p})$, with coefficient in $\QQ_{p}$. Then $\A_{p} \cong \QQ_{p}[U]$ and $\A_{p}^{-} \cong \QQ_{p}[U^{-}]$(see \cite[Prop. 6.4.1]{Bellaiche.2009}). Let us define $\HH := \A_{p} \otimes \HH^{ur}$ to be the Hecke Algebra. Let us also define $\HH^{-} := \A_{p}^{-} \otimes \HH^{ur}$.\\

\noindent The weight space $\W_n$ is the rigid analytic space over $\Q_{p}$ defined as $\W_n := \mathrm{Hom}_{group,cts}(T^{0}, \mathbb{G}_{m}^{rig})$. It is isomorphic to finite disjoint union of unit open $n$-dimensional balls. We see $\Z^n$ embedded in $\W_n$ via the map $ (k_1, \cdots, k_n) \mapsto ((x_1,\cdots,x_n) \mapsto x_{1}^{k_{1}}\cdots x_{n}^{k_{n}}) $. Let $V \subset \W_{n}$  be an affinoid open. Then Chenevier defined the space of $p$-adic automorphic forms of weight in $V$, radius of convergence $r$, $\mathcal{S}(V,r)$( \cite[Section 7.3.4]{Bellaiche.2009})as follows :
Let $I$ denote the Iwahori subgroup of $GL_{n}(\Q_{p})$, and $B$ the standard Borel subgroup. Let $\bar{N}_{0}$ be the subgroup of lower triangular matrices of $I$. Let $\chi: T^{0} \to \OO(\W)^{\ast}$ denote the tautological character and $\chi_{V}: T^{0} \to A(V)^{\ast}$ the induced continuous character. For $r \geq r_{V}$, define the space of functions:
$$\C(V,r) = \{ f: IB \to A(V) \mid f(xb)= \chi_{V}(b)f(x), \forall x \in IB, b\in B, f|_{\bar{N}_{0}} \text{ is } r \text{ analytic} \}.$$ 
$C(V,r)$ is a $M:= IU^{-}I$ module. For a $M$-module $E$, define an $\HH^{-}$-module $F(E)$ as:
$$F(E) = \{ f: G(\Q) \backslash G(\mathbb{A}_{f}) \to E \mid f(g(1 \times k_{p})) = k_{p}^{-1}f(g), \forall g \in G(\mathbb{A}_{f}), k_{p} \in I, f \text{ is smooth outside }p \}.$$ 
We define the space of $p$-adic automorphic forms of weights in $V$, radius of convergence $r$ as the $\HH^{-}$-module $\SSS(V,r):= F(\C(V,r))$. $\SSS(V,r)$ is an ortho-normalizable Banach $A(V)$-module with property (Pr) (as in \cite{Buzzard.2007}). It is equipped with an $A(V)$ linear action of $\HH^{-} = \A_{p}^{-} \otimes \HH^{ur}$. The modules $\{ \mathcal{S}(V,r),V,r \geq r_{V} \}$ are compatible in a sense defined there.\\
\noindent If $\underline{k}$ is a point in the weight space, the above construction remains valid, we get a module $M_{(\underline{k})} = \SSS(\underline{k},r)$. Moreover if $\underline{k} \in \Z^{n,-}$, the choice of an highest-weight vector in $W_{\underline{k}}(\Q_{p})$ with respect to the standard Borel $B$ gives a natural $\HH^{-}$-equivariant inclusion:
$$ F(W_{\underline{k}}(L)^{\ast}) \otimes \delta_{\underline{k}} \hookrightarrow \SSS(\underline{k},0), $$
where $\delta_{\underline{k}}$ is a character of $U$ obtained by restricting $W_{\underline{k}}(\Q_{p})$ to $U$. Image of the map is referred as the space of classical $p$-adic automorphic forms of weight $\underline{k}$. \\

\noindent Let $u_0 = diag(p^{n-1},p^{n-2},\cdots,p,1)$. Then $u_0$ acts compactly on the ({\bf Pr})-family of Banach modules $\{ \mathcal{S}(V,r),V,r \geq r_{V} \}$. Let $P=P_{u_{0}}$ be the Fredholm series associated to $u_0$. Then for $r \geq r_{V}$, we have,
\begin{equation*}
P(T)_{|V} = det (1- Tu_{0|\mathcal{S}(V,r)}) \in 1+ TA(V)\{\{T\}\}.
\end{equation*}
Fix a $\nu \in \R$ such that $V$ ($P(T)$) is adapted to $\nu$ (see \cite[Definition II.1.8]{CourseBook}, also see \cite{Buzzard.2007}). Then in $A(V)\{\{T\}\}$, $P$ has an unique factorization $P = QR$, with $Q \in 1+ TA(V)[T]$ and we have a decomposition of $A(V)$ Banach module $\SSS(V,r)$ as,
\begin{equation}\label{module}
 \mathcal{S}(V,r) = M(V)^{\leq \nu} \oplus N(V,r)
\end{equation}
\noindent which is $\HH^{-}$ stable and such that:
\begin{itemize}
\item $M(V)^{\leq \nu}$ is a finite projective $A(V)$-module which is independent of $r \geq r_{V}$.
\item The characteristic polynomial of $u_0$ on $M(V)^{\leq \nu}$ is $Q^{rec}(T)$, the reciprocal polynomial on $Q$, and $Q^{rec}(u_{0})$ is invertible on $N(V,r)$.
\end{itemize} 

\noindent In particular the eigenvalues of $U_{p}$ on $M(V)^{\leq \nu}$ has valuation less than or equal to $\nu$. When $V$ is a point $\underline{k}$, we similarly get a finite projective module $M_{\underline{k}}^{\leq \nu}$.\\

\noindent The local piece of the eigenvariety is by definition the maximal spectrum of the $A(V)$-algebra generated by the image of $\HH = \HH^{-}[u_0]^{-1}$ in $ End_{A(V)}(M(V)^{\leq \nu})$. By Buzzard's eigenvariety machine (see \cite{Buzzard.2007}) these pieces glue together to form the eigenvariety (for details see \cite[Section 7.3.6]{Bellaiche.2009}), which we will denote by $\E_{n}$. The eigenvariety is reduced $p$-adic analytic space $\E_{n}$, equipped with a ring homomorphism $\psi: \HH \to \OO(\E_{n})^{rig}$, and an analytic map $\omega: \E_{n} \to \W$.

\subsection{Points on the eigenvariety}

Let $\underline{k} = (k_1, \cdots, k_n) \in \Z^{n}$ such that $k_1 \geq k_2 \geq \cdots \geq k_n$.
A $p$-refined automorphic representation of weight $\underline{k}$ is a tuple $( \pi, \RR)$ such that,
\begin{itemize}
\item $\pi$ is an irreducible automorphic representation of $G$.
\item $\pi_{\infty} \simeq_{\iota_{\infty}} W_{\underline{k}}(\CC)$.
\item $\pi_{p}$ is unramified and $\RR$ is an accessible refinement of $\pi_{p}$.\cite[Section 6.4.4]{Bellaiche.2009}
\end{itemize}
Let us recall that an refinement of $\pi_{p}$ is an ordering 
\begin{equation*}
\RR = (\phi_1,\cdots,\phi_{n})
\end{equation*}
of the eigenvalues of the Langlands conjugacy associated to $\pi_{p}$. Equivalently, it is a character $\chi: U \to \CC^{\ast}$, sending $(1,\cdots,1,p,1,\cdots,1)$ to $\phi_{i}$. It is called accessible if $\chi \delta_{B}^{-1/2}$ occurs in $\pi_{p}^{I}$, here $B$ is the standard Borel subgroup, $I$ is the Iwahori subgroup and $\delta_{B}$ is the modulus character.

\noindent A $p$-refined automorphic representation of weight $\underline{k}$ is a tuple $( \pi, \RR)$, which is viewed as a point in the eigenvariety in the following manner:\\
Let $\delta_{\underline{k}} : U \to p^{\Z}$ be the character sending $diag(u_1,\cdots,u_{n}) \to u_{1}^{k_{1}}\cdots u_{n}^{k_{n}}$. Then there is a unique ring homomorphism $\psi_{p}: \A_{p} \to \CC$, such that 
\begin{equation}\label{psip}
\psi_{p|U} = \chi \delta_{B}^{-1/2} \delta_{\underline{k}}.
\end{equation}
Moreover we also have a ring homomorphism, $\psi_{ur}: \HH^{ur} \to \CC$, since $\pi^{G(\hat{\Z}_{S_{0}})}$ is one dimensional. Hence we have a system of $\HH$ eigenvalue $\psi_{p} \otimes \psi_{ur}$. It is in-fact $\QQ$-valued. To the $p$-refined automorphic representation $(\pi, \RR)$ of weight $\underline{k}$, we associate a $\QQ_{p}$valued system of Hecke eigenvalues $\psi_{(\pi,\RR)} = \iota_{p}\iota_{\infty}^{-1}(\psi_{p} \otimes \psi_{ur})$.\\

\noindent Let $\mathcal{Z}_{0} \subset \mathrm{Hom}_{ring}(\HH,\QQ_{p}) \times \Z^{n}$ be the set of pairs $(\psi_{(\pi,\RR)} ,\underline{k})$ associated to all $p$-refined automorphic form $(\pi,\RR)$ of weight $\underline{k}$. For a fixed $\mathcal{Z} \subset \mathcal{Z}_{0}$, the associated eigenvariety $\E_{n}$ comes with an accumulation and Zariski dense subset $Z \subset \E_{n}(\QQ_{p})$, such that the natural evaluation map $\E_{n}(\QQ_{p}) \to  \mathrm{Hom}_{ring}(\HH,\QQ_{p})$, sending $x \mapsto \psi_{x} := (h \mapsto \psi(h)(x))$, induces a bijection between $Z$ and $\mathcal{Z}$ sending $z \mapsto (\psi_{z}, \omega(z))$. (We might think $\E_{n}$ as ``Zariski closure" of the set $\mathcal{Z}$.)\\

\subsubsection{Control Theorem}
Let $\Z^{n,-} = \{(a_1,\cdots,a_n)| a_{1} > \cdots > a_{n} \}$. Let $\underline{k} = (k_1,\cdots,k_n) \in \Z^{n,-}$.
For this section let $V = \underline{k}$, be a closed point. Let $U^{--} = \{diag(p^{a_1},\cdots,p^{a_{n}})| a_{i} \in \Z, a_1 > \cdots > a_{n} \}$.
An element $ f \in \mathcal{S}(\underline{k},r)$ is of finite slope if for some (hence for all)$u \in U^{--}$ ,
\begin{itemize}
\item $L[u].f \subset  \mathcal{S}(\underline{k},r)$ is finite dimensional
\item $u_{|L[u].f} $ is invertible.
\end{itemize}
Hence the finite slope elements form an $L$-subspace $\mathcal{S}(\underline{k},r)^{fs}$ of $\mathcal{S}(\underline{k},r)$, and the $\A_{p}^{-}$ module structure on $\mathcal{S}(\underline{k},r)^{fs}$ extends to $\A_{p}$ module structure \cite[Section 7.3.5]{Bellaiche.2009}. Note that, $M_{\underline{k}}^{\leq \nu}$ as in (\ref{module}) , is contained in $\mathcal{S}(\underline{k},r)^{fs}$. We have a natural inclusion of $\HH$-modules $ F(W_{\underline{k}}(L)^{\ast}) \otimes \delta_{\underline{k}} \subset \SSS(\underline{k},r)^{fs} \subset \SSS(\underline{k},0) $, for all $r \in \mathbb{N}$.\\

\noindent The following proposition allows us to determine when a finite slope overconvergent eigenform is classical.
\begin{prop}\label{chencontrolthm}\cite[Proposition 7.3.5]{Bellaiche.2009}\cite[Proposition 4.7.4]{Chenevier.2004}
Let $f \in \SSS(\underline{k},r)^{fs} \otimes_{L} \QQ_{p}$ be an eigenform for all $u \in U \subset \A_{p}$. Let $u_{0} = diag(p^n-1,\cdots,p,1)$, write $u_{0}(f) = \lambda f$ for $\lambda \in \QQ_{p}^{\ast}$. If we have,
\begin{equation}
v(\lambda) < 1+ min_{i=1}^{n-1} (k_{i} - k_{i+1}),
\end{equation}
then $f$ is classical.
More precisely, under the same condition the full generalized $\A_{p}$-eigenspace of $f$ in $\SSS(\underline{k},r)^{fs} \otimes_{L} \QQ_{p}$ is included in the subspace $F(W_{\underline{k}}(L)^{\ast}) \otimes_{L} \QQ_{p}$ of classical forms.
\end{prop}
 
\section{Endoscopic Transfer for automorphic forms on $U(n)$}

In this section we recall known (and expected) results about classical endoscopic transfer for automorphic forms on the definite unitary group $U(n)$.
 
\subsection{Admissible map of L-groups of $U(n)$}

Let $E$ be a CM extension of $\Q$ in $\bar{\mathbb{Q}}$. Let $Gal(E/\Q) = \{1, \sigma \}$. Let $\omega_{E/\Q}$ denote the character of order $2$ of idele class group $C_{\Q}$ associated to $E/\Q$ by CFT\footnote{Let $N_{E/\Q}: C_E \to C_{\Q}$ denotes the norm map, so its image is index $2$ subgroup of $C_{\Q}$. Then $\omega_{E/\Q}$ is the nontrivial character of $C_{\Q}$ which is trivial on the image $N_{E/\Q}(C_E)$.}.\\
The L-group of $U(n)$ is ${^L}{U(n)} = GL_n(\mathbb{C}) \rtimes W_{\Q}$, where the Weil group $W_{\Q}$ acts on $GL_{n}(\CC)$ by its quotient $Gal(E/\Q)$ by $\sigma M \sigma^{-1} := \phi_{n} (M^{-1})^{t} \phi_{n}^{-1}$, where $M \in GL_{n}(\CC)$ and $(\phi_{n})_{i,j} = (-1)^{n+i}\delta_{i,n+1-j}$.\\
The L-group of $U(n_1)\times \cdots \times U(n_{r})$ is ${^L}{(U(n_1) \times \dots \times U(n_r))} = (GL_{n_1}(\CC) \times \dots \times GL_{n_r}(\CC)) \rtimes W_{\Q}$, where the Weil group $W_{\Q}$ acts on $GL_{n_1}(\CC) \times \dots \times GL_{n_r}(\CC)$ by its quotient $Gal(E/\Q)$.\\
 
\noindent Let $n= n_1+\cdots+n_r$ be an unordered partion. Following Rogawski, we have an admissible map of L-groups \cite[Section 1.2]{Rogawski.1992} as below:
 
\begin{equation}\label{lembd}
 \xi:{^L}{(U(n_1) \times \dots \times U(n_r))} \to {^L}{U(n)}  
\end{equation}

\noindent which maps the neutral component of ${^L}{(U(n_1) \times \dots \times U(n_r))} = GL_{n_1} \times \dots \times GL_{n_r}$ to the subgroup of diagonal blocks of size $n_1,\dots,n_r$. We fix for once and all a character $\mu$ of the idele class group $C_E$ whose restriction to $C_{\Q}$ is $\omega_{E/\Q}$. We regard $\mu$ as a character of $W_E$ by means of isomorphism $W_{E}^{ab} \to C_E$. Set
\begin{equation}\label{mu}
 \mu_j = 
 \begin{cases}
 \mu & \text{if } n_j \equiv n-1 \text{ mod }2, \\
 1 & \text{otherwise}.
 \end{cases}
\end{equation}

\noindent and for $w \in W_E$, let 
\begin{equation*}
\xi(w) = \xi( \mu_{1}(w)I_{n_1},\dots,\mu_{r}(w) I_{n_r}) \times w.
\end{equation*}

\noindent To extend $\xi$ to $W_{\Q}$, it suffices to define $\xi(w_{\sigma})$, where $w_{\sigma} \in W_{\Q}$ is a fixed element whose projection to $Gal(E/F)$ is $\sigma$. We define
\begin{equation*}
\xi(w_{\sigma}) = \xi( \Phi_{n_1},\dots,\Phi_{n_r}) \Phi_{n}^{-1} \times w_{\sigma},
\end{equation*}
where $(\Phi_n)_{ij} = (-1)^{i-1} \delta_{i,n-j+1}$.\\

\noindent This map is a special case of endoscopic functorality as $U(n_{1}) \times \cdots \times U(n_{r})$ is not a Levi subgroup of $U(n)$ if $r>1$. Also the map $\xi$ depends on the choice of character $\mu$, it is unique if one fixes a character $\mu$. From now onwards, we fix a character $\mu$ once and for all.

\subsection{Endoscopic Transfer for unramified representation of $U(n)$}
 Endoscopic transfer is a special case of Langland's functoriality. If $n = n_1+ \cdots + n_r$, then $H =U(n_1) \times \cdots \times U(n_r)$ is an endoscopic subgroup of $U(n)$. The map $\xi$ defined as in (\ref{lembd}), defines the map of corresponding L-groups. Using $\xi$, we transfer parameters of $H$ to that of $U(n)$, we call it endoscopic transfer. \\
 
\noindent Let us now restrict to the unramified parameter. We start with an unramified parameter $\psi_l$ for $H$. So the parameter $\psi_l$  is trivial on the the inertia group, hence factor through $\Z$, and is determined by the image of Frobenius $\Phi$. We compose this with the map $\xi$, to get an unramified parameter for $U(n)$.
 \begin{equation}\label{unrendotrnsfr}
\xymatrix{
W_{\Q_{l}} \times SL_2(\CC) \ar[rr]^{\psi_{l}} \ar@{->>}[d] &  &{^L}{U(n_1) \times \cdots \times U(n_r)} \ar[rr]^{\xi} & & {^L}{U(n)} \\
W_{\Q_{l}}/ I_{\Q_{l}} \ar[r]^{\cong} & \Z \ar[ur]_{\psi_{l}} & & &
}
\end{equation}
\noindent Since unramified parameters are in bijection with unramified representations\cite[Proposition 1.12.1]{MR1265563}, we get endoscpoic transfer for unramified representation of $U(n)$.

\subsection{Global Endoscopic Transfer for $U(n)$}

Let $F$ be a number field and $L_F $ denotes the conjectural Langlands group\cite{Langlands.1979}.
\begin{lem}\cite[Lemma 2.2.1]{Rogawski.1992}
Let $\psi$ be a discrete parameter for $U(n)$, that is a morphism $\psi: L_{\Q} \times SL_{2}(\mathbb{C}) \to {^L}{U(n)}$.
Then $\psi_{E} : L_{E} \times SL_2(\mathbb{C}) \to GL_{m}(\mathbb{C})$ is a sum of pairwise nonisomorphic irreducible representation of ${\rho}_{j}$ such that ${\rho}_{j}^{\perp} \cong {\rho}_{j}$.
\end{lem}
\noindent This lemma allows us to define stable A-parameter.
 \begin{defn}
We call a discrete A-parameter $\psi$ stable if the base change to $\psi_{E}$ is irreducible.
An A-parameter which is not stable is called endoscopic.
\end{defn}

\noindent Let $ n = n_{1} + \dots + n_{r}$ be a partition  and for $j = 1,\dots, r $, let $\psi_{j}$ be a parameter of $U(n_j)$ of the form $\psi_{j}(\gamma \times h) = {\alpha}_{j}(\gamma \times h) \times w_{\gamma} $. We denote the product parameter $ \psi_{1} \times \dots \times \psi_{r}$ for $U(n_1)\times \dots \times U(n_r)$ by \\
\begin{equation*}
\gamma \times h \to ({\alpha}_{1}(\gamma \times h) \times \dots \times {\alpha}_{r}(\gamma \times h)) \times w_{\gamma}.
\end{equation*}
The equivalence class of $ \psi_{1} \times \dots \times \psi_{r}$ is independent of the ordering of $\psi_{j}$.\\

\noindent Following lemma describes endoscopic parameter for $U(n)$.

\begin{lem}\cite[Lemma 2.2.2]{Rogawski.1992}\label{etransApara}
Let $\psi$ be a discrete parameter for $U(n)$. Then there is a unique (unordered) partition $ n = n_{1} + \dots + n_{r}$ and distinct stable parameters $\psi_{j}$ for $U(n_j)$ such that $\psi = \xi \circ ( \psi_{1} \times \dots \times \psi_{r}).$ \\
Conversely, if  $ n = n_{1} + \dots + n_{r}$ and $\psi_{j}$ is a stable parameter for $U(n_j)$, then $\xi \circ ( \psi_{1} \times \dots \times \psi_{r})$ is a discrete parameter for $U(n)$ if and only if the $\psi_{j}$ are distinct.
Here $\xi$ is the map defined in \ref{lembd}.
\end{lem}

\noindent In view of previous lemma, a discrete parameter of $U(n)$  can be uniquely as  $\psi = \xi \circ ( \psi_{1} \times \dots \times \psi_{r})$, where $\psi_{j}$ are distinct stable parameter for $U(n_{j})$. So a discrete parameter of $U(n)$ is endoscopic if $r>1$.\\

\noindent We make the following assumptions regarding $A$-packets of $U(n)$. 

\begin{itemize}
\item A packets are defined for $U(n)$.
\item $\Pi_{1}$, $\Pi_2$ discrete A-packets such that for almost all $v$, $\Pi_{1,v}$ and $\Pi_{2,v}$ contain the same unramified representations, then $\Pi_{1}$ and $\Pi_{2}$ coincide.
\end{itemize}
\begin{defn}\label{stableApacket}
Let $\Pi = \bigotimes_{v} \Pi_{v}$ be a discrete A-packet for $U(n)$. Then $\Pi$ is stable if there exists a discrete representation  $\pi = \bigotimes_{v} \pi_{v}$  of $GL_{n/E}$ such that for all $v$ for which $\Pi_{v}$ is unramified, the base change $(\Pi_{v})_{E}$ coincide with $\pi_{v}$(For unramified representation, base change map is well defined).
If $\pi$ is cuspidal, we call $\Pi$ cuspidal.
An A-packet which is not stable is called Endoscopic.
\end{defn}

\noindent Arthur's conjecture predicts existence of a natural correspondence which associates to every global discrete $A$-parameter of $U(n)$ (upto equivalence) an $A$-packet of $U(n)$, or the empty set. Under the conjecture stable(resp. endoscopic) $A$-parameter of $U(n)$ corresponds to a stable(resp. endoscopic) $A$-packet of $U(n)$. Predictions of Arthur's conjecture for $U(n)$ were verified by Rogawski when $n \leq 3$ \cite{MR1081540}. Main obstruction to prove the existence of global Endoscopic transfer for $U(n)$ and arthur conjecture for $U(n)$ for $n>3$ was the proof of fundamental lemma. Since by the work of Ng{\^o}, fundamental lemma is proved, a group of mathematicians are working on writing down the details of endoscopic transfer of unitary group and verification of Arthur's conjecture.\\
 
\noindent We will assume that the global endoscopic transfer for $U(n)$ exists, in fact one can assume that the global endoscopic transfer for $U(n)$ exists for almost all automorphic representations. This is expected to be true by the work of Closel, Harris, Labesse, Ng{\^o}, and many other mathematicians (see \cite{HarrisBook1}, volume 2 under preparation) .

\section{Maps between eigenvarieties of $U(n)$ interpolating endoscopic transfer}
 
Let $\E_{n}$, $\E_{n_{1}}$ and $\E_{n_{2}}$ denote the eigenvariety associated to $U(n)$, $U(n_{1})$ and $U(n_{2})$ respectively, where $n = n_{1}+ n_{2}$. In this section we will construct a map $f: \E_{n_{1}} \times \E_{n_{2}} \to \E_{n}$, which at classical points interpolates endoscopic transfer. But first we prove two simple lemmas regarding system of eigenvalues, which will be used during the construction.\\
 
\begin{lem}\label{sevlifting}

Let $M$ be a finite dimensional vector space over $\QQ_{p}$. Let $H \subset H^{\prime}$ be two commutative $\QQ_{p}$ algebras acting on $M$. Then any $H$ system of eigenvalue appearing in $M$, extends to a $H^{\prime}$ system of eigenvalue appearing in $M$.

\end{lem}

\begin{proof}
 Let $\chi$ be a $H$ system of eigenvalue appearing in $M$. Let $M[\chi]$ be the corresponding eigenspace.
 Let $m_1 \in M[\chi]$ , $ h \in H$, $ h^{\prime} \in H^{\prime}$, then 
 \begin{equation*}
 h.(h^{\prime}.m_1) = h^{\prime}. (h. m_1) = h^{\prime}. (\chi(h)m_1) = \chi(h) (h^{\prime}.m_1)
 \end{equation*}
 So $h$ acts on $h^{\prime}.m_1$ by $\chi(h)$. Hence $h^{\prime}.m_1 \in M[\chi]$. 
 So $M[\chi]$ is stable under the action of $H^{\prime}$. Since for any commutative algebra acting on a finite dimensional vector space over an algebraically closed field has a common eigenvector, we see there exists a $m \in M[\chi]$, a common eigenvector for the action of $H^{\prime}$ on $M[\chi]$. Hence $m$ is an eigenvector for the action of $H^{\prime}$ on $M$. So $\chi$ extends to a $H^{\prime}$ system of eigenvalue in $M$.\end{proof}

\noindent {\bf Note:} A system of $H^{\prime}$ eigenvalue which extends a given system of $H$ eigenvalue need not be unique.

\begin{lem}\label{sevtensor}
Let $H_1$, $H_2$ be two $\QQ_{p}$ algebras acting semisimply on a finite dimensional $\QQ_{p}$ vector space $M_1$ , $M_2$ via $\psi_1$, $\psi_2$ respectively. Then a system of $H_1 \otimes H_2$ eigenvalue appearing in $M_1 \otimes M_2$ is same as a system of $H_1$ eigenvalue appearing in $M_1$ and a system of $H_2$ eigenvalue appearing in $M_2$ composed via tensor product.
\end{lem}

\begin{proof}
Let $\chi_1$ be a system of $H_1$ eigenvalue for $M_1$ and $\chi_2$ be a system of $H_2$ eigenvalue for $M_2$, that is, there exists $v_1 \in M_1$, such that, $\chi_1 : H_1 \to \QQ_{p} , \chi_1(h_1)v_1 = \psi_1(h_1)v_1$ for all $h_1 \in H_1$ , and there exists $v_2 \in M_2$, such that, $\chi_2 : H_2 \to \QQ_{p} , \chi_2(h_2)v_2 = \psi_2(h_2)v_2$ for all $h_2 \in H_2$.\\
\noindent We want to show that, $\chi = \chi_1 \otimes \chi_2 : H_1 \otimes H_2 \to \QQ_{p}$ defined by
 $\chi(h_1 \otimes h_2) = \chi_1(h_1)\chi_2(h_2)$ is a system of eigenvalue for $H_{1} \otimes H_{2}$ appearing in $M_{1} \otimes M_{2}$.
\noindent Let us define $\psi = \psi_1 \otimes \psi_2$. Let $ h = \sum_{i}{\alpha_{i}(h_{1,i} \otimes h_{2,i}) }$ be an element of $H_1\otimes H_2$. Then we have,
 \begin{eqnarray*}
   \psi(h) (v_1 \otimes v_2) & = & \sum_{i} {\alpha_{i} \psi(h_{1,i} \otimes h_{2,i})(v_1 \otimes v_2)} \\
   & = & \sum_{i} {\alpha_{i} \psi_1(h_{1,i})v_1 \otimes \psi_2(h_{2,i})v_2}  \\
   & = & \sum_{i} {\alpha_{i} \chi_1(h_{1,i})v_1 \otimes \chi_2(h_{2,i})v_2}  \\
   & = & \chi(\sum_{i}{\alpha_{i}(h_{1,i} \otimes h_{2,i}) })(v_1 \otimes v_2)  \\ 
   & = & \chi(h)(v_1 \otimes v_2).
 \end{eqnarray*}
\noindent  Hence $\chi$ is a system of eigenvalue for $H_1 \otimes H_2$ appearing in $M_1 \otimes M_2$.\\
  \vskip 2mm  
\noindent Let $\chi$ be a system of eigenvalue for $H_1 \otimes H_2$ appearing in $M_1 \otimes M_2$. We want show that $\chi = \chi_{1} \otimes \chi_{2}$ for some $\chi_1$ (resp. $\chi_2$) a system of eigenvalue for $H_1$ (resp. $H_2$) appearing in $M_1$ (resp. $M_2$). Since $H_i$ acts semi-simply on $M_i$, we have,
  \begin{equation*}
  M_1 \cong \oplus_{i=1}^{r} M_1[\chi_{1,i}],
 \end{equation*}
   \begin{equation*}
  M_2 \cong \oplus_{j=1}^{s} M_2[\chi_{2,j}] .
 \end{equation*}
\noindent Let $ \mathrm{dim } M_1[\chi_{1,i}] = n_i$ generated by $\{ a_1, \cdots, a_{n_{i}} \}$, and $ \mathrm{dim } M_2[\chi_{2,j}] = m_j$ generated by $\{ b_1, \cdots, b_{m_{j}} \}$. Then we have,
 \begin{equation*}
\mathrm{dim } M_1[\chi_{1,i}] \otimes M_2[\chi_{2,j}] = n_i m_j = (\mathrm{dim } M_1[\chi_{1,i}])(\mathrm{dim } M_2[\chi_{2,j}]),
 \end{equation*}
 generated by $\{ a_k \otimes b_l \}_{k= 1,\cdots,n_i ; l=1,\cdots,m_j}$. By previous part, $\chi_{1,i}\otimes \chi_{2,j}$ is a system of $H_1 \otimes H_2$ eigenvalue for $M_1 \otimes M_2$ with $a_k \otimes b_l$ as an eigenvector. Hence we have, $M_1\otimes M_2 [\chi_{1,i} \otimes \chi_{2,j}] \supseteq M_1[\chi_{1,i}] \otimes M_2[\chi_{2,j}].$ We see that,
 \begin{eqnarray*} 
 dim M_1\otimes M_2 [\chi_{1,i} \otimes \chi_{2,j}] & \geq & dim M_1[\chi_{1,i}] \otimes M_2[\chi_{2,j}] ,\\
 & = & (dim M_1[\chi_{1,i}])(dim M_2[\chi_{2,j}]). 
 \end{eqnarray*}
 Summing over all $i$ and $j$, we get,
 \begin{eqnarray*} 
 \sum_{1=1}^{r}{ \sum_{j=1}^{s}{dim M_1\otimes M_2 [\chi_{1,i} \otimes \chi_{2,j}] }} & \geq & \sum_{1=1}^{r}{ \sum_{j=1}^{s}{(dim M_1[\chi_{1,i}])(dim M_2[\chi_{2,j}]) }}, \\
 & = & (dim M_1) (dim M_2), \\
 & = & dim (M_1  \otimes M_2).
\end{eqnarray*} 
Thus we have,
\begin{equation*}  
 M_1  \otimes M_2 \cong \oplus_{i,j} M_1  \otimes M_2 [\chi_{1,i} \otimes \chi_{2,j}].
\end{equation*} 
\noindent Hence any system of eigenvalue $\chi$ of $H_1 \otimes H_2$ in $M_1 \otimes M_2$ is of the form $\chi_{1,i} \otimes \chi_{2,j}$.
\end{proof}

\noindent A key ingredient to construct the map is a version of comparison theorem due to Chenevier\cite{MR2111512}, which we recall below.\\
\noindent Let us fix an eigenvariety data, a ring $\HH$ with a distinguished element $U_{p}$, $\W$ a reduced rigid space with an admissible covering $\mathfrak{C}$, and Banach modules $M_{W}$ with an action of $\HH$ for all $W \in \mathfrak{C}$, which satisfies compatibility criterion.

\begin{defn}\label{clstr}\cite[Definition II.5.4]{CourseBook}
A classical structure on eigenvariety data is the data of\\
(CSD1) a very Zariski dense subset $X \subset \W$,\\
(CSD2) for every $x \in X$, a finite dimensional $\HH$ module $M_{x}^{cl}$\\
such that\\
(CSC1) for every $x \in X$, there exists an $\HH$ equivariant injective map $M_{x}^{cl,ss} \hookrightarrow M_{x}^{fs,ss}$\\
(CSC2) for every $\nu \in \R$, let $X_{\nu}$ be the set of $x \in X$, such that there exists an $\HH$-equivariant isomorphism $M_{x}^{cl, \leq \nu} \cong M_{x}^{fs, \leq \nu}$, then for every $x \in X$, there exists a basis of neighborhoods of $x \in V$, such that $X_{\nu} \cap V$ is Zariski dense in $V$.
\end{defn}

\begin{thm}{Comparison Theorem(Chenevier)} \label{compthm}\cite[Theorem II.5.6]{CourseBook}
Suppose that we have two eigenvariety data with same $\HH,\W,\mathfrak{C}$, but different Banach modules $M_{W}$ and $M_{W}^{\prime}$. Let us call $\E$ and $\E^{\prime}$ the two eigenvarieties attached to those data. Assume that the two eigenvarieties are each provided with a classical structure with the same set $X$ (CSD1), but different $\HH$ modules $M_{x}^{cl}$ and $M_{x}^{\prime,cl}$. Suppose that for every $x \in X$, there exists an $\HH$-equivariant injective map
$$ M_{x}^{\prime,cl,ss} \hookrightarrow M_{x}^{cl,ss}.$$
Then there exists a unique closed embedding $\E^{\prime} \hookrightarrow \E$ of the eigenvarieties compatible with weight maps to $\W$ and with the maps $\HH \to \OO(\E)$ and $\HH \to \OO(\E^{\prime})$.
\end{thm}

\noindent Let $u_{i} = diag(1,\cdots,1,p,1,\cdots,1) $, where $p$ occurs at $i$-th place. Let $F_{i} = \psi(u_{i})$, where $\psi: \HH \to \OO(\E_{n})$.  Let $\mathcal{Z}(n)$ be the subset of $x \in \E_{n}(\QQ_{p})$ such that,\\
(a) $\omega(x) = (k_{1}, k_{2},\cdots ,k_{n}) \in \Z^{n,-}$,\\
(b) $v(F_{1}(x)F_{2}(x) \cdots F_{n-1}(x)) < 1+ min_{i=1}^{n-1} \text{ }(k_{i} - k_{i+1})$, \\
(c)  if $\phi_{i} := F_{i}(x) p^{- k_{i} + i -1} $, then for all $i \neq j $, $\phi_{i}\phi_{j}^{-1} \neq p$.\\
Then $\mathcal{Z}(n)$ is a Zariski accumulation dense subset of $\E_{n}$ \cite[Theorem 7.3.1]{Bellaiche.2009}.
Let $\mathcal{Z}_{reg}(n)$ be the subset of $\mathcal{Z}(n)$, parameterizing the $p$-refined $(\pi,\RR)$ such that $\pi_{\infty}$ is regular, and such that the semisimple conjugacy class of $\pi_{p}$ has $n$ distinct eigenvalues. Then $\mathcal{Z}_{reg}(n)$ is a Zariski dense subset of $\E_{n}$ accumulating at each point of $\mathcal{Z}(n)$ \cite[Lemma 7.5.3]{Bellaiche.2009}.

\noindent Let $\kappa = (\kappa_{1},\cdots,\kappa_{n}): \E_{n} \to \mathbb{A}^n$ is the composition of the map $log_{p} \circ \omega$ by the affine change of co-ordinates $(x_{1},\cdots,x_{n}) \mapsto (-x_{1},-x_{2}+1,\cdots,-x_{n}+n-1)$. Then for $z=(\pi,\RR) \in Z_{reg}$, with $\omega(z) = (k_{1}, k_{2},\cdots ,k_{n})$, we have $\kappa_{i}(z) = -k_{i}+i-1$, and $\kappa_{1}(z),\cdots,\kappa_{n}(z)$ is the strictly increasing sequence of Hodge-Tate weights of $\rho_{\pi}$, where $\rho_{\pi}$ is the Galois representation associated to $\pi$ \cite[Definition 7.5.11]{Bellaiche.2009}. Moreover, we have \cite[Definition 7.2.13]{Bellaiche.2009},
 
 $$i_{p}i_{\infty}^{-1} (\RR|p|^{\frac{1-n}{2}} ) =  (F_{1}(z)p^{\kappa_{1}(z)},\cdots, F_{i}(z)p^{\kappa_{i}(z)}, \cdots, F_{n}(z)p^{\kappa_{n}(z)}).$$

\noindent Let $z_{1} \in \mathcal{Z}_{reg}(n_{1})$ be any point with $\omega(z_{1})= (k_{1},\cdots,k_{n_{1}}) \in \Z^{n_{1},-}$, and $z_{2} \in \mathcal{Z}_{reg}(n_{2})$ be any point with $\omega(z_{2})= (k_{1}^{\prime},\cdots,k_{n_{2}}^{\prime}) \in \Z^{n_{2},-}$. We call the pair $(z_{1},z_{2}) \in \mathcal{Z}_{reg}(n_{1}) \times \mathcal{Z}_{reg}(n_{2})$ a ``nice" classical point if we have $k_{n_{1}} > k_{1}^{\prime}$.

\begin{thm}{{\bf Endoscopic Transfer on Eigenvariety for $U(n)$}}
There exists a unique map of eigenvarieties 
$$f: \E_{n_{1}} \times \E_{n_{2}} \to \E_{n},$$
where $n = n_{1}+ n_{2}$, such that a ``nice" classical point of $\E_{n_{1}} \times \E_{n_{2}}$, $((\pi_{1}, \RR_{1}),(\pi_{2},\RR_{2}))$ maps to a classical point of $\E_{n}$, $(\pi,\RR)$, where $\pi$ is the endoscopic transfer of $\pi_1$ and $\pi_2$ and $\RR = (\mu_{1}(Frob_{p})\RR_{1}, \mu_{2}(Frob_{p})\RR_{2})$, here $\mu_1$ and $\mu_2$ as in (\ref{mu}).
\end{thm}

\begin{proof}
({\bf Uniqueness})
Let $\phi^{\prime}, \phi^{\prime \prime}: \E_{n_{1}} \times \E_{n_{2}} \to \E_{n}$ be two such maps interpolating endoscopic transfer at ``nice" classical points. Let $z \in  \E_{n_{1}} \times \E_{n_{2}} $ be a ``nice" classical point, that is $z = ((\pi_{1},\RR_{1} ),(\pi_{2},\RR_{2}))$, where $\pi_{i}$ is a $p$-refined automorphic form for $U(n_{i})$ of weight $\underline{k_{i}}$, where $\underline{k_{1}} =(k_{1}, \cdots, k_{n_{1}})$ and $\underline{k_{2}} =(k_{n_{1}+1}, \cdots, k_{n})$ with $k_{1}>k_{2}> \cdots > k_{n}$. Then $\pi$, the endoscopic transfer of $\pi_{1}$ and $\pi_{2}$  has weight $\underline{k} = (k_{1},k_{2}, \cdots , k_{n})$. Suppose $\phi^{\prime} (z) = (\pi, \RR^{\prime})$ and $\phi^{\prime \prime}(z) = (\pi, \RR^{\prime\prime})$, where $\RR^{\prime},\RR^{\prime\prime}$ are any two accessible refinements of $\pi_{p}$. Let $\RR^{\prime\prime} = \sigma_{z}(\RR^{\prime})$, where $\sigma_{z} \in S_{n}$, as $\RR^{\prime\prime}$ is an ordering of $\RR^{\prime}$. The set of ``nice" classical points can be written as,
$$Z = \bigsqcup_{\sigma \in S_{n}}\{ z \in Z | \sigma_{z}= \sigma\} := \bigsqcup_{\sigma \in S_{n}} Z_{\sigma}.$$
\noindent Since $Z$ is dense, at least one of the $Z_{\sigma}$ is dense. Suppose $Z_{\sigma_{0}}$ is dense, we will use $z \in Z_{\sigma_{0}}$, hence $\sigma_{0}(\RR^{\prime}) = \RR^{\prime\prime}$, for all $z \in Z_{\sigma_{0}}$.\\
\noindent Note that,
 \begin{equation*}
 i_{p}i_{\infty}^{-1} (\RR^{\prime}|p|^{\frac{1-n}{2}} ) =  (F_{1}(\phi^{\prime}(z))p^{\kappa_{1}(\phi^{\prime}(z))},\cdots, F_{i}(\phi^{\prime}(z))p^{\kappa_{i}(\phi^{\prime}(z))}, \cdots, F_{n}(\phi^{\prime}(z))p^{\kappa_{n}(\phi^{\prime}(z))}),
 \end{equation*}
and
 \begin{equation*}
 i_{p}i_{\infty}^{-1} (\RR^{\prime\prime}|p|^{\frac{1-n}{2}} ) =  (F_{1}(\phi^{\prime\prime}(z))p^{\kappa_{1}(\phi^{\prime\prime}(z))},\cdots, F_{i}(\phi^{\prime\prime}(z))p^{\kappa_{i}(\phi^{\prime\prime}(z))}, \cdots, F_{n}(\phi^{\prime\prime}(z))p^{\kappa_{n}(\phi^{\prime\prime}(z))}).
 \end{equation*}

\noindent Thus we have, 
$$F_{\sigma_{0}(i)} (\phi^{\prime}(z)) p ^{\kappa_{\sigma_{0}(i)}(\phi^{\prime}(z))} = F_{i}(\phi^{\prime \prime}(z)) p^{\kappa_{i}(\phi^{\prime \prime}(z))},$$
and hence we have, 
$$ F_{i}(\phi^{\prime \prime}(z)) = F_{\sigma_{0}(i)}(\phi^{\prime}(z)) p^{\kappa_{\sigma_{0}(i)}(\phi^{\prime}(z))- \kappa_{i}(\phi^{\prime \prime}(z))}.$$
Since $\kappa_{i}(\phi^{\prime}(z))= \kappa_{i}(\phi^{\prime \prime}(z))$ for all $i$, as $\kappa_{i}$ is the strictly increasing sequence of Hodge-Tate weights of the Galois representation $\rho_{\pi}$, it is independent of the choice of the refinement. We have,
$$ F_{i}(\phi^{\prime \prime}(z)) = F_{\sigma_{0}(i)}(\phi^{\prime}(z)) p^{\kappa_{\sigma_{0}(i)}(\phi^{\prime}(z))- \kappa_{i}(\phi^{\prime}(z))},$$ 
and hence,
$$  F_{i}(\phi^{\prime \prime}(z)) = F_{\sigma_{0}(i)}(\phi^{\prime}(z)) p^{(\kappa_{\sigma_{0}(i)}- \kappa_{i})(\phi^{\prime}(z))}.$$
The left hand side of the equation is an analytic function, but the right hand side is not analytic unless $\sigma_{0}(i) = i$ for all $i$. Thus $\RR^{\prime} = \RR^{\prime\prime}$. Hence $\phi^{\prime}(z) = \phi^{\prime \prime}(z)$ for all $z \in Z_{\sigma_{0}}$. Since $Z_{\sigma_{0}}$ is dense, we get $\phi^{\prime} = \phi^{\prime \prime}$.\\

\noindent {\bf Existance:}
Let $\HH_{1}$, $\HH_2$ and $\HH$ be Hecke algebras associated to $U(n_1)$, $U(n_2)$ and $U(n)$ respectively. We have, $\HH_{1} \cong \otimes_{l \in S_{0}} \HH_{1,l} \otimes \A_{p,1}$, where $\HH_{1,l} \cong \QQ_{p}[X_{1,l},\cdots,X_{n_{1},l},X_{1,l}^{-1},\cdots,X_{n_{1},l}^{-1}]^{S_{n_{1}}}$ and $\A_{p,1} \cong \QQ_p[X_1,\cdots,X_{n_{1}},X_{1}^{-1},\cdots,X_{n_{1}}^{-1}] $. Similarly, we have, $\HH_{2} \cong \otimes_{l \in S_{0}} \HH_{2,l} \otimes \A_{p,2}$, where $\HH_{2,l} \cong \QQ_{p}[Y_{1,l},\cdots,Y_{n_{2},l},Y_{1,l}^{-1},\cdots,Y_{n_{2},l}^{-1}]^{S_{n_{2}}}$ and $\A_{p,2} \cong \QQ_p[Y_1,\cdots,Y_{n_{2}},Y_{1}^{-1},\cdots,Y_{n_{2}}^{-1}] $ and $\HH \cong \otimes_{l \in S_{0}} \HH_{l} \otimes \A_{p}$, where $\HH_{l} \cong \QQ_{p}[T_{1,l},\cdots,T_{n,l},T_{1,l}^{-1},\cdots,T_{n,l}^{-1}]^{S_{n}}$ and $\A_{p} \cong \QQ_p[T_1,\cdots,T_{n},T_{1}^{-1},\cdots,T_{n}^{-1}] $.\\

\noindent We have an injection of $\HH_{l}$ into $ \HH_{1,l} \otimes \HH_{2,l}$ given by, 
\begin{equation}\label{mapatl}
\theta: \HH_{l} \hookrightarrow \HH_{1,l} \otimes \HH_{2,l}
\end{equation}

\begin{equation*}
\theta(T_{i,l}) = 
 \begin{cases}
 \mu_{1}(Frob_{l})X_{i,l} & \text{if } 1 \leq i \leq n_{1}, \\
 \mu_{2}(Frob_{l})Y_{i- n_{1},l} & \text{if } n_{1} < i \leq n,
 \end{cases}
\end{equation*}
where $\mu_{i}$ as in (\ref{mu}), and $Frob_{l}$ is the frobenius. \\

\noindent Also we have an isomorphism of Atkin-Lehner algebra given by, 
\begin{equation}\label{mapatp}
\A_p \cong \A_{p,1} \otimes \A_{p,2}
\end{equation}
\begin{equation*}
T_i \mapsto
\begin{cases}
\mu_1(Frob_{p})p^{\frac{n_{2}}{2}} X_i & \text{if } 1 \leq i \leq n_{1}, \\
 \mu_{2}(Frob_{p}) p^{\frac{-n_{1}}{2}} Y_{i- n_{1}} & \text{if } n_{1} < i \leq n.
\end{cases}
\end{equation*}
Combining these two maps, we get an injection $G: \HH \hookrightarrow \HH_{1} \otimes \HH_{2}$.\\
Notice that $\W_{n_{1}} \times \W_{n_{2}} \cong \W_{n}$. Let $W \subset \W_n$ be an affinoid open such that $W \cong W_1 \times W_2$, and $W_{i} \subset \W_{n_{i}}$ is an affinoid open. Let $R$ be the affinoid algebra so that $ Spec(R) = W $. Let $M_{W} =\SSS(W,r)$ and $M_{W}^{\prime} = M_{W_1} \otimes M_{W_2}=\SSS(W_{1},r) \otimes \SSS(W_{2},r)$, with $r \geq max(r_{W},r_{W_{1}},r_{W_{2}})$. Then $M_{W}$ is a $\HH$ module and $M_{W}^{\prime}$ is a $\HH_{1} \otimes \HH_{2}$ module and hence a $\HH$ module. Using $(R,M_{W},\HH)$, we construct the local piece of eigenvariety $\E_{n}$, using $(R,M_{W}^{\prime},\HH_{1} \otimes \HH_{2})$ we construct local piece of eigenvariety $ \E_{n_{1}} \times \E_{n_{2}}$. Let us denote by $\E^{\prime}$ the eigenvariety whose local pieces are constructed from the data $(R, M_{W}^{\prime}, \HH)$. Since $\HH \hookrightarrow \HH_{1} \otimes \HH_{2}$, there is a surjective map $ \alpha: \E_{n_{1}} \times \E_{n_{2}} \to \E^{\prime}$. So to construct a map $f: \E_{n_{1}} \times \E_{n_{2}} \to \E_{n}$, we need to construct a map $\beta : \E^{\prime} \to \E_{n}$.\\

\noindent To construct the map $\beta$, we use the comparison theorem due to Chenevier (\ref{compthm}). We first need to define a classical structure on $\E$ and $\E^{\prime}$. For $x \in \W(L)$, we shall denote by $(M_{x})^{fs}$ (resp. $(M_{x}^{\prime})^{fs}$), the finite slope part of the fiber at $x$ of the eigenvariety data used to construct $\E$ (resp. $\E^{\prime}$), base changed to $\QQ_{p}$. So if we have $x = \underline{k}=(k_{1},\cdots,k_{n}) \in \Z^{n,-}$, then, $(M_{\underline{k}})^{fs} = \SSS(\underline{k},r)^{fs} \otimes \QQ_{p}$ and $(M_{\underline{k}}^{\prime})^{fs} = (\SSS(\underline{k_{1}},r)^{fs} \otimes \QQ_{p}) \otimes (\SSS(\underline{k_{2}},r)^{fs} \otimes \QQ_{p})$, where $\underline{k_{1}} = (k_{1},\cdots,k_{n_{1}})$ and $\underline{k_{2}} = (k_{n_{1}+1},\cdots,k_{n})$. In both cases take $X=\Z^{n,-} \subset \W$ (CSD1). For $x = \underline{k}= (\underline{k_{1}},\underline{k_{2}})=(k_{1},\cdots,k_{n}) \in \Z^{n,-}$, define $(M_{\underline{k}})^{cl} = (F(W_{\underline{k}}(L)^{\ast})^{fs} \otimes \delta_{\underline{k}}) \otimes \QQ_{p}$, $(M_{\underline{k}}^{\prime})^{cl} = ((F(W_{\underline{k_{1}}}(L)^{\ast})^{fs}\otimes \delta_{\underline{k_{1}}})  \otimes \QQ_{p}) \otimes ((F(W_{\underline{k_{2}}}(L)^{\ast})^{fs} \otimes \delta_{\underline{k_{2}}})  \otimes \QQ_{p})$. Since we have natural $\HH$-equivariant embedding $F(W_{\underline{k}}(L) \otimes \delta_{\underline{k}} \hookrightarrow \SSS(\underline{k},r)^{fs}$, we have as $\HH$ module $(M_{\underline{k}})^{cl}$ (resp. $(M_{\underline{k}}^{\prime})^{cl}$), are submodule of $(M_{\underline{k}})^{fs}$ (resp. $(M_{\underline{k}}^{\prime})^{fs}$), hence condition (CSC1) is satisfied. For condition (CSC2), fix $\nu \in \R$. Define, 
$$X_{\nu} = \{ \underline{k} \in \Z^{n,-} | k_{i} - k_{i+1} + 1 > \nu \text { for all } i = 1,\cdots, (n-1) \}.$$
Thus by Chenevier's control theorem (\ref{chencontrolthm}) we have, for any $\underline{k} \in X_{\nu}$, $(M_{\underline{k}})^{cl, \leq \nu} \simeq (M_{\underline{k}})^{fs,\leq \nu} $ and $(M_{\underline{k}}^{\prime})^{cl, \leq \nu} \simeq (M_{\underline{k}}^{\prime})^{fs,\leq \nu} $. Hence the condition (CSC2) is satisfied for all cases. Having defined classical structure on $\E$ and $\E^{\prime}$, we need to construct an $\HH$-equivariant injective map $(M_{x}^{\prime})^{cl, ss} \hookrightarrow (M_{x})^{cl, ss}$ for all $x \in X$, to get the map $\beta: \E^{\prime} \hookrightarrow \E$. Thus we want to show, every $\HH$ system of eigenvalue appearing in $(M_{x}^{\prime})^{cl, ss}$ also appears as an $\HH$ system of eigenvalue in $(M_{x})^{cl, ss}$.\\

\noindent Let $\chi$ be a $\HH$ system of eigenvalue appearing in $(M_{\underline{k}}^{\prime})^{cl, ss}$. By lemma \ref{sevlifting} we can lift $\chi$ to $\tilde{\chi}$ a $\HH_{1} \otimes \HH_{2}$ system of eigenvalue in $(M_{\underline{k_1}}^{\prime})^{cl, ss}$. Then by lemma \ref{sevtensor} $\tilde{\chi} = \chi_{1} \otimes \chi_{2}$, where $\chi_{i}$ is a $\HH_{i}$ system of eigenvalue appearing in $((F(W_{\underline{k_{1}}}(L)^{\ast})^{fs}\otimes \delta_{\underline{k_{i}}})  \otimes \QQ_{p})^{ss} $. So $\chi_{i}$ corresponds to $(\pi,\RR_{i})$, where $\pi_{i}$ is an automorphic representation of $U(n_{i})$ of weight $\underline{k_{i}}$ and $\RR_{i}$ is an accessible refinement of $\pi_{i,p}$. Let $\pi$ be the endoscopic transfer of $\pi_{1}$ and $\pi_{2}$. Let $\RR = (\mu_{1}(Frob_{p})\RR_{1}, \mu_{2}(Frob_{p})\RR_{2})$. Let $\chi^{\prime}$ be the $\HH$ system of eigenvalue in $(M_{\underline{k}})^{cl, ss}$ corresponding to $(\pi,\RR)$. We want to show $\chi^{\prime} = \chi $. To prove this it is enough to show $\chi_{p}^{\prime} = \chi_{p}$ and $\chi_{l}^{\prime} = \chi_{l}$ for $l \neq p$.

\begin{lem}\label{localendoatl}
$\chi_{l}^{\prime} =  \chi_{l}$, where $\chi_{l}$ and $\chi_{l}^{\prime}$ are as described above.
\end{lem}

\begin{proof}
Let $\pi_{i,l}$ be the unramified representation of $GL_{n_{i}}$ corresponding to $\chi_{i,l}$ Since $\pi_{1,l}$ and $\pi_{2,l}$ are unramified representaions of $GL_{n_{1}}$ and $GL_{n_{2}}$ respectively,they are determined by their Satake parameters say $(\lambda_{1,l},\cdots,\lambda_{n_{1},l})$ and $(\lambda_{1,l}^{\prime},\cdots, \lambda_{n_{2},l}^{\prime})$ respectively. \\
 Under the endoscopic transfer map, $\pi_{1,l}$ and $\pi_{2,l}$ maps to a unramified representation of $GL_{n}$ as follows:\\
 \begin{adjustbox}{max size = {\textwidth}}
 $
 \xymatrix{
 \Z \ar[r] &  {^L}{U(n_1)\times U(n_2)} \ar[r]^{\xi} &  {^L}{U(n)} \\ 
 Frob_{l} \ar@{|->}[r] & ( diag(\lambda_{1,l},\cdots,\lambda_{n_{1},l}), diag(\lambda_{1,l}^{\prime},\cdots, \lambda_{n_{2},l}^{\prime})) \times Frob_{l} \ar@{|->}[r] &  diag(\lambda_{1,l},\cdots,\lambda_{n_{1},l},\lambda_{1,l}^{\prime},\cdots, \lambda_{n_{2},l}^{\prime}) \xi(Frob_{l})
 }
$
 \end{adjustbox}
Since $ \xi(Frob_{l}) = \xi( \mu_{1}(Frob_{l})I_{n_1},\mu_{2}(Frob_{l}) I_{n_2}) \times Frob_{l} $, we have under endoscopic transfer as in (\ref{unrendotrnsfr}),
 \begin{equation*}
  Frob_{l} \mapsto diag(\mu_{1}(Frob_{l})\lambda_1, \cdots,\mu_1(Frob_{l})\lambda_{n_{1}},\mu_2(Frob_{l})\lambda_{1}^{\prime},\cdots,\mu_2(Frob_{l}) \lambda_{n_{2}}^{\prime}) \times Frob_{l} .
 \end{equation*}
 Hence if $\pi_{l}$ is the endoscopic transfer of $ \pi_{1,l}$ and $\pi_{2,l}$, then the Satake parameter for $\pi_{l}$ is given by $(\mu_{1}(Frob_{l})\lambda_{1,l}, \cdots,\mu_1(Frob_{l})\lambda_{n_{1},l},\mu_2(Frob_{l})\lambda_{1,l}^{\prime},\cdots,\mu_2(Frob_{l}) \lambda_{n_{2},l}^{\prime})$. We want to show $\pi_{l}$ corresponds to $\chi_{l}$ in $M_{x}$. 
 Since Satake parameter of $\pi_{1,l}$ is $(\lambda_{1,l},\cdots,\lambda_{n_{1},l})$, which corresponds to $\chi_{1,l}$, we have,
 \begin{equation*}
 \chi_{1,l} : H_{1,l} \to \QQ_{p} 
 \end{equation*}
 \begin{equation*}
 \chi_{1,l}(s_{k,1,l}) = \sum_{1 \leq i_{1} < \cdots < i_{k} \leq n_{1}} \lambda_{i_{1},l}\cdots \lambda_{i_{k},l},
 \end{equation*}
 where $s_{k,1}$ is symmetric polynomial in $X_{1,l}, \cdots, X_{n_{1},l}$ of degree $k$. Similarly we have,
 \begin{equation*}
 \chi_{2,l}(s_{k,2,l}) = \sum_{1 \leq j_{1} < \cdots < j_{k} \leq n_{2}} \lambda_{j_{1},l}^{\prime} \cdots \lambda_{j_{k},l}^{\prime},
 \end{equation*}
 where $s_{k,2,l}$ is symmetric polynomial in $Y_{1,l}, \cdots, Y_{n_{2},l}$ of degree $k$.
 
\noindent  By the work of Satake, $\QQ_{p}[X_{1},\cdots,X_{n_{1}},X_{1}^{-1},\cdots,X_{n_{1}}^{-1}]$ is integral over $\QQ_{p}[X_{1},\cdots,X_{n_{1}},X_{1}^{-1},\cdots,X_{n_{1}}^{-1}]^{S_{n_{1}}}$, and any character of $\QQ_{p}[X_{1},\cdots,X_{n_{1}},X_{1}^{-1},\cdots,X_{n_{1}}^{-1}]^{S_{n_{1}}}$ can be lifted to a character of the ring $\QQ_{p}[X_{1},\cdots,X_{n_{1}},X_{1}^{-1},\cdots,X_{n_{1}}^{-1}]$. Hence,  $\chi_{1,l}$ and $\chi_{2,l}$ can be extended to  character  $\tilde{\chi}_{1,l}$ and $\tilde{\chi}_{2,l}$ of $\QQ_{p}[X_{1},\cdots,X_{n_{1}},X_{1}^{-1},\cdots,X_{n_{1}}^{-1}]$ and $\QQ_{p}[Y_{1},\cdots,Y_{n_{2}},Y_{1}^{-1},\cdots,Y_{n_{2}}^{-1}]$ respectively. \\
 Let us assume $\tilde{\chi}_{1,l}(X_{i,l}) = \lambda_{i,l} $ and $\tilde{\chi}_{2,l}(Y_{j,l}) = \lambda_{j,l}^{\prime}$. \\
 
\noindent Let $\tilde{\chi_{l}} = \tilde{\chi}_{1,l} \otimes \tilde{\chi}_{2,l}$, then
 \begin{equation*}
 \tilde{\chi_{l}}(T_{i,l}) =
 \begin{cases}
 \tilde{\chi}_{1,l}(\mu_{1}(Frob_{l}))X_{i,l}) = \mu_1(Frob_{l}) \lambda_{i,l} & \text{if } 1 \leq i \leq n_{1}, \\
 \tilde{\chi}_{2,l}(\mu_{2}(Frob_{l})Y_{i- n_{1},l}) = \mu_2(Frob_{l}) \lambda_{i-n_{1},l}^{\prime} & \text{if } n_{1} < i \leq n.
 \end{cases}
 \end{equation*}
Note that $\tilde{\chi}_{l|\HH_{1,l} \otimes \HH_{2,l}}  = \chi_{1,l} \otimes \chi_{2,l}$ and $\tilde{\chi}_{l|H} = \chi_{l} $. Hence $\chi_{l}$ corresponds to $\pi_{l}$, the endoscopic transfer of $\pi_{1,l}$ and $\pi_{2,l}$.
This is independent of the choice of $\tilde{\chi}_{1,l}$ and $\tilde{\chi}_{2,l}$ as even if we choose different lifts, the values of $X_{i,l}$, $Y_{j,l}$, $s_{k,1,l}$,$s_{k,2,l}$ will be different, but the values of $s_{k,l}$, the symmetric polynomial in $T_{1,l}, \cdots,T_{n,l}$, will be same. And hence the Satake parameter corresponding to $\chi_{l}$ will be same. 
Since by definition of $\chi^{\prime}$, $\chi_{l}^{\prime}$ corresponds to $\pi_{l}$, we have $\chi_{l}^{\prime} = \chi_{l}$.
\end{proof}

\begin{lem}
$\chi_{p}^{\prime} =  \chi_{p}$, where $\chi_{p}$ and $\chi_{p}^{\prime}$ are defined as above.
\end{lem}

\begin{proof}
By the lemma \ref{sevtensor} , every system of $\A_p$ eigenvalue $\chi_p$ in $(M_{\underline{k}}^{\prime})^{cl,ss}$, is of the form $\chi_{p,1} \otimes \chi_{p,2}$, where $\chi_{p,i}$ is an $\A_{p,i}$ system of eigenvalue appearing in $(M_{\underline{k_{1}}}^{\prime})^{cl,ss}$. So $\chi_{p,i}$ corresponds to $(\pi_{p,i},\RR_{i})$, where $\pi_{p,i}$ is an unramified representation of $GL_{n_{i}}$ at p, and $\RR_{i}$ is an accessible refinement. Then as in equation (\ref{psip}), we have

\begin{equation}
\chi_{p,i|U} = \gamma_{i} \delta_{B_i}^{-1/2} \delta_{\underline{k_i}},
\end{equation}
\noindent where $B_i$ is the borel subgroup of $GL_{n_i}$,$\delta_{B_i}$ is the modulus character, $\gamma_i: U_i \to \C^{\ast} $ corresponds to the refinement $\RR_i$ and $\delta_{\underline{k_i}}$ as in (\ref{psip}). \\

\noindent {\bf Note:1} Since $ \underline{k} = (\underline{k_1},\underline{k_2})$, we have the following relation:
\begin{equation}\label{delta}
\delta_{\underline{k_1}}(diag(p^{\alpha_1},\cdots,p^{\alpha_{n_{1}}}))\delta_{\underline{k_2}}(diag(p^{\alpha_{n_{1}+1}},\cdots,p^{\alpha_n})) = \delta_{\underline{k}}(diag(p^{\alpha_1},\cdots,p^{\alpha_n}))
\end{equation}

\noindent {\bf Note : 2} We have the following relation for the modulus character,
 
 \begin{eqnarray*}
 \delta_{B}^{-1/2}(diag(p^{\alpha_1},\cdots,p^{\alpha_n})) & = & |p^{\alpha_1}|^{\frac{1-n}{2}} \cdots |p^{\alpha_{n_{1}}}|^{\frac{2n_{1}-1-n}{2}}|p^{\alpha_{n_{1}+1}}|^{\frac{2n_{1}+1-n}{2}} \cdots |p^{\alpha_n}|^{\frac{n-1}{2}} \\
 &=&  \delta_{B_1}^{-1/2}(diag(p^{\alpha_1},\cdots,p^{\alpha_{n_{1}}})) \delta_{B_2}^{-1/2}(diag(p^{\alpha_{n_{1}+1}},\cdots,p^{\alpha_n})) \\
 & & (p^{\alpha_1} \cdots p^{\alpha_{n_{1}}})^{\frac{n_{2}}{2}}(p^{\alpha_{n_{1}+1}} \cdots p^{\alpha_n})^{-\frac{n_{1}}{2}}. 
 \end{eqnarray*}
Hence we have, 
 \begin{equation}\label{modulus}
 \delta_{B_1}^{-1/2}(diag(p^{\alpha_1},\cdots,p^{\alpha_{n_{1}}})) \delta_{B_2}^{-1/2}(diag(p^{\alpha_{n_{1}+1}},\cdots,p^{\alpha_n})) p^{\frac{n_{2}}{2} (\alpha_1 + \cdots + \alpha_{n_{1}})- \frac{n_{1}}{2}(\alpha_{n_{1}+1} + \cdots + \alpha_n)} 
 \end{equation}
 \begin{equation*}
 = \delta_{B}^{-1/2}(diag(p^{\alpha_1},\cdots,p^{\alpha_n}))
 \end{equation*}
 
 \noindent For $\alpha_{i} \in \Z$, we have,
 \begin{eqnarray*}
 \chi_p(diag(p^{\alpha_1},\cdots,p^{\alpha_n})) &= & \psi_p(T_{1}^{\alpha_{1}}\cdots T_{n}^{\alpha_{n}}) \\
   & =&  \chi_{p,1} \otimes \chi_{p,2} (T_{1}^{\alpha_{1}}\cdots T_{n}^{\alpha_{n}})  \\
   &= &  \chi_{p,1} \otimes \chi_{p,2} ((\mu_1(Frob_{p})p^{\frac{n_{2}}{2}} X_1)^{\alpha_{1}} \cdots (\mu_{2}(Frob_{p}) p^{-\frac{n_{1}}{2}} Y_{n_{2}})^{\alpha_{n}} )  \\
   &= &  \chi_{p,1} \otimes \chi_{p,2} (\mu_1(Frob_{p})^{\alpha_{1}+\cdots+\alpha_{n_{1}}}\mu_2(Frob_{p})^{\alpha_{n_{1}+1}+\cdots+\alpha_{n}} \\
    & & p^{\frac{n_{2}}{2} (\alpha_1 + \cdots + \alpha_{n_{1}})- \frac{n_{1}}{2}(\alpha_{n_{1}+1} + \cdots + \alpha_n)}X_{1}^{\alpha_1}\cdots X_{n_{1}}^{\alpha_{n_{1}}}Y_{1}^{\alpha_{n_{1}+1}}\cdots Y_{n_{2}}^{\alpha_{n}} ) \\
    &= & \mu_1(Frob_{p})^{\alpha_{1}+\cdots+\alpha_{n_{1}}}\mu_2(Frob_{p})^{\alpha_{n_{1}+1}+\cdots+\alpha_{n}} p^{\frac{n_{2}}{2} (\alpha_1 + \cdots + \alpha_{n_{1}})- \frac{n_{1}}{2}(\alpha_{n_{1}+1} + \cdots + \alpha_n)}  \\
    &  &\chi_{p,1}(X_{1}^{\alpha_1}\cdots X_{n_{1}}^{\alpha_{n_{1}}}) \chi_{p,2}(Y_{1}^{\alpha_{n_{1}+1}}\cdots Y_{n_{2}}^{\alpha_{n}}) \\
   & =& \mu_1(Frob_{p})^{\alpha_{1}+\cdots+\alpha_{n_{1}}}\mu_2(Frob_{p})^{\alpha_{n_{1}+1}+\cdots+\alpha_{n}} p^{\frac{n_{2}}{2} (\alpha_1 + \cdots + \alpha_{n_{1}})- \frac{n_{1}}{2}(\alpha_{n_{1}+1} + \cdots + \alpha_n)} \\         &  & \chi_{p,1}(diag(p^{\alpha_1},\cdots,p^{\alpha_{n_{1}}})) \chi_{p,2}(diag(p^{\alpha_{n_{1}+1}},\cdots ,p^{\alpha_{n}})) \\
   & =& \mu_1(Frob_{p})^{\alpha_{1}+\cdots+\alpha_{n_{1}}}\mu_2(Frob_{p})^{\alpha_{n_{1}+1}+\cdots+\alpha_{n}} p^{\frac{n_{2}}{2} (\alpha_1 + \cdots + \alpha_{n_{1}})- \frac{n_{1}}{2}(\alpha_{n_{1}+1} + \cdots + \alpha_n)} \\         &  & \gamma_{1}(diag(p^{\alpha_1},\cdots,p^{\alpha_{n_{1}}})) \gamma_{2}(diag(p^{\alpha_{n_{1}+1}},\cdots ,p^{\alpha_{n}})) \\     
    &  &  \delta_{B_1}^{-1/2}(diag(p^{\alpha_1},\cdots,p^{\alpha_{n_{1}}}))\delta_{B_2}^{-1/2}(diag(p^{\alpha_{n_{1}+1}},\cdots,p^{\alpha_n}))\\
    &  & \delta_{\underline{k_1}}(diag(p^{\alpha_1},\cdots,p^{\alpha_{n_{1}}})) \delta_{\underline{k_2}}(diag(p^{\alpha_{n_{1}+1}},\cdots,p^{\alpha_n})) \\
    & = & \mu_1(Frob_{p})^{\alpha_{1}+\cdots+\alpha_{n_{1}}}\mu_2(Frob_{p})^{\alpha_{n_{1}+1}+\cdots+\alpha_{n}} \gamma_{1}(diag(p^{\alpha_1},\cdots,p^{\alpha_{n_{1}}})) \\
    & & \gamma_{2}(diag(p^{\alpha_{n_{1}+1}},\cdots ,p^{\alpha_{n}}))  \delta_{B}^{-1/2}(diag(p^{\alpha_1},\cdots,p^{\alpha_n}))\\
    &  & \delta_{\underline{k}}(diag(p^{\alpha_1},\cdots,p^{\alpha_n}))  \\
    & = & \gamma \delta_{B}^{-1/2} \delta_{\underline{k}} (diag(p^{\alpha_1},\cdots,p^{\alpha_n})),
\end{eqnarray*}
where $\gamma : U \to \CC^{\ast}$ is the character given by,
 \begin{eqnarray*}\label{endorefinement}
 \gamma(diag(p^{\alpha_1},\cdots,p^{\alpha_n})) & = &\mu_1(Frob_{p})^{\alpha_{1}+\cdots+\alpha_{n_{1}}}\mu_2(Frob_{p})^{\alpha_{n_{1}+1}+\cdots+\alpha_{n}} \gamma_{1}(diag(p^{\alpha_1},\cdots,p^{\alpha_{n_{1}}}))\\ 
 &  & \gamma_{2}(diag(p^{\alpha_{n_{1}+1}},\cdots ,p^{\alpha_{n}})).
 \end{eqnarray*}
Since $\chi_{p|U}$ is of the form $\gamma \delta_{B}^{-1/2} \delta_{\underline{k}}$, it appears as a system of eigenvalue for $\A_{p}$ in $M_{x}$ and the refinement it corresponds to is given by $\gamma$.\\

\noindent If $\RR_1 = (\phi_1,\cdots,\phi_{n_{1}})$ and $\RR_2 = (\phi_1^{\prime},\cdots,\phi_{n_{2}}^{\prime})$, 
then $\phi_{i} = \gamma_{1}(diag(1,\cdots,1,p,1,\cdots,1))$ and $\phi_{j}^{\prime} = \gamma_{2}(diag(1,\cdots,1,p,1,\cdots,1))$. Thus we have,
\begin{equation*}
\gamma(diag(1,\cdots,1,p,1,\cdots,1)) =
\begin{cases}
\gamma_{1}(diag(1,\cdots,1,p,1,\cdots,1) \mu_{1}(Frob_{p})  & \text{if } 1 \leq i \leq n_{1}, \\
\gamma_{2}(diag(1,\cdots,1,p,1,\cdots,1) \mu_{2}(Frob_{p})  & \text{if } n_{1}< i \leq n.
\end{cases}
\end{equation*}
Hence we see that,
\begin{equation*}
\chi(diag(1,\cdots,1,p,1,\cdots,1)) =
\begin{cases}
\phi_{i} \mu_{1}(Frob_{p}) & \text{if } 1 \leq i \leq n_{1}, \\
\phi_{i-n_{1}}^{\prime} \mu_{2}(Frob_{p}) & \text{if } n_{1}< i \leq n.
\end{cases}
\end{equation*}

\noindent So $\gamma$ corresponds to the refinement given by,
\begin{eqnarray*}
\RR &=& (\phi_1 \mu_1(Frob_{p}), \cdots, \phi_{n_{1}} \mu_1(Frob_{p}), \phi_{1}^{\prime} \mu_2(Frob_{p}), \cdots, \phi_{n_{2}}^{\prime} \mu_2(Frob_{p}))\\
  &=& (\mu_1((Frob_{p})\RR_{1}, \mu_2(Frob_{p})\RR_2)).
\end{eqnarray*}
We see that, $\chi_{p}$  corresponds to $(\pi_{p},\RR)$ , where $\pi_{p}$ is the endoscopic transfer of $\pi_{p,1}$ and $\pi_{p,2}$, since as in the proof of the lemma \ref{localendoatl}, $\pi_{p}$ the endoscopic transfer of $\pi_{p,1}$ and $\pi_{p,2}$ has Langlands parameter given by 
$(\phi_1 \mu_1(Frob_{p}), \cdots, \phi_{n_{1}} \mu_1(Frob_{p}), \phi_{1}^{\prime} \mu_2(Frob_{p}), \cdots, \phi_{n_{2}}^{\prime} \mu_2(Frob_{p}))$. By definition of $\chi^{\prime}$, $\chi_{p}^{\prime}$ corresponds to $(\pi_{p}, \RR)$. Thus $\chi_{p}^{\prime} = \chi_{p}$.

\end{proof}

\noindent Thus we obtain a map, $f = \beta \circ \alpha : \E_{n_{1}} \times \E_{n_{2}} \to \E_{n}$. Let $((\pi_{1},\RR_{1}), (\pi_{2},\RR_{2}))$ be a point in $\E_{n_{1}} \times \E_{n_{2}}$, which corresponds to $\HH_{1}\otimes \HH_{2}$ system of eigenvalue $\chi_{1} \otimes \chi_{2}$ in $(M_{\underline{k}}^{\prime})^{cl}$ . Under the map $\alpha$, $\chi_{1} \otimes \chi_{2}$ maps to $\HH$ system of eigenvalue $\chi = (\chi_{1} \otimes \chi_{2})_{|\HH}$ in $(M_{\underline{k}}^{\prime})^{cl}$. Under the map $\beta$, $\chi$ maps to $\HH$ system of eigenvalue $\chi^{\prime} = \chi$ appearing in $(M_{\underline{k}})^{cl}$, which corresponds to $(\pi,\RR)$, where $\pi$ is the endoscopic transfer of $\pi_{1}$ and $\pi_{2}$ and $\RR = (\mu_1((Frob_{p})\RR_{1}, \mu_2(Frob_{p})\RR_2))$. 
\vskip 5mm

\end{proof}
\noindent
{\bf Acknowledgments.}\\

\noindent I would like to take this opportunity to thank professor Joel Bellaiche for being my advisor and for spending so much time and energy to share his great insights with me. I would also like to thank him for suggesting the problem and give various ideas to help me solving the problem. I would like to thank Baskar Balasubramanyam and Debargha Banerjee for their help during preparation of the manuscript. Most of the work was done during author's thesis at Brandeis University under the supervision of professor Joel Bellaiche. \vskip
5mm

\bibliographystyle{amsplain}
\bibliography{endoscopy}

\vskip 5mm

\noindent {\bf Address of author:} \vskip 5mm

\noindent {\it Dipramit Majumdar}, Indian Institute of Science Education and Research, Pune, Dr Homi Bhaba Road, Pashan, Pune 411008, India.\\
{\em dipramit@gmail.com}

\end{document}